\documentclass[reqno,12pt]{amsart}
\usepackage{color} 
\usepackage{ifpdf}
\ifpdf
    \usepackage[pdftex]{graphicx}
    \usepackage[pdftex]{hyperref}
    \hypersetup{
        unicode=false,          
        pdftoolbar=true,        
        pdfmenubar=true,        
        pdffitwindow=false,     
        pdfstartview={FitH},    
        pdftitle={YZ Article},      
        pdfauthor={Yunrong Zhu},   
        pdfsubject={Mathematics},    
        pdfcreator={Yunrong Zhu},  
        pdfproducer={Yunrong Zhu}, 
        pdfkeywords={PDE, numerical analysis, preconditioner, multigrid}, 
        pdfnewwindow=true,      
        colorlinks=true,        
        linkcolor=red,          
        citecolor=blue,         
        filecolor=magenta,      
        urlcolor=cyan           
    }

\else
    \usepackage{graphicx}
    \usepackage{pstricks}
    
    \newcommand{\href}[2]{#2}
\fi

\usepackage{times}
\usepackage{amsfonts}
\usepackage[font={small}]{caption}

\usepackage{amsmath}
\usepackage{amsthm}
\usepackage{amssymb}
\usepackage{amsbsy}
\usepackage{amscd}

\usepackage{enumerate}
\usepackage{verbatim}
\usepackage{subfigure}

\newtheorem{theorem}{Theorem}[section]

\newtheorem{lemma}[theorem]{Lemma}

\newtheorem{assumption}[theorem]{Assumption}

\newtheorem{remark}[theorem]{Remark}

\numberwithin{equation}{section}  

\renewcommand{\P}{{\mathbb P}}
\newcommand{\Q}{{\mathbb Q}}
\newcommand{\R}{{\mathbb R}}

\newcommand{\bs}{\boldsymbol}

\newcommand{\cT}{{\mathcal T}}
\def\l{{\langle}}
\def\r{{\rangle}}
\renewcommand\lll{|\kern-1pt|\kern-1pt|} 

\newcommand{\inn}{\text{ in }}
\newcommand{\ale}{\text{a.e.\,}} 

\def\tbar{|\hspace*{-0.15em}|\hspace*{-0.15em}|}

\begin{document}

\title[Stabilizer-Free WG for Monotone PDEs]{Stabilizer-Free Weak Galerkin Methods for Monotone Quasilinear Elliptic PDEs}

\author[X. Ye]{Xiu Ye}
\email{xxye@ualr.edu}

\author[S. Zhang]{Shangyou Zhang}
\email{szhang@udel.edu}

\author[Y. Zhu]{Yunrong Zhu}
\email{zhuyunr@isu.edu}

\address{Department of Mathematics\\
         University of Arkansas at Little Rock\\
         Little Rock, AR 72204.}
\thanks{This research was supported in part by National Science Foundation Grant DMS-1620016}

\address{Departmetn of Mathematical Sciences\\
	 University of Delaware\\
	  Newark, DE 19716.}

\address{Department of Mathematics and Statistics\\
         Idaho State University\\
         Pocatello, ID 83209.}
\thanks{YZ was supported in part by NSF DMS 1319110.}

\date{\today}



\begin{abstract}
	In this paper, we study the stabilizer-free weak Galerkin methods on polytopal meshes for a class of second order elliptic boundary value problems of divergence form and with gradient nonlinearity in the principal coefficient. With certain assumptions on the nonlinear coefficient, we show that the discrete problem has a unique solution. This is achieved by showing that the associated operator satisfies certain continuity and monotonicity properties. With the help of these properties, we derive  optimal error estimates in the energy norm. We present several numerical examples to verify the error estimates.
\end{abstract}

\maketitle




\section{Introduction}\label{sec:intro}
We consider the stabilizer free weak Galerkin method for  the quasilinear elliptic partial
differential equations (PDEs) in $\Omega\subset \R^d$ with $d=2,$ or $3$:
\begin{equation}\label{eqn:model}
\left\{\begin{array}{rll}
-\nabla \cdot (\kappa(x,|\nabla u|) \nabla u) &= f &~\inn~ \Omega,\\
u &=0 &~\text{ on } \partial \Omega.
\end{array}\right.
\end{equation}
This class of PDEs arise, for example, in the study of compressible flow in the airfoil design (cf. \cite{Glowinski.R1984}) or the eddy currents in a nonlinear ferromagnetic material (cf. \cite{Seidman1985}).  We assume throughout the paper that the diffusion coefficient $\kappa(x,s)$, for $x \in \Omega$ and $s \in \R$ satisfies the following assumption.
\begin{assumption}\label{A1:elliptic}
	Assume $\kappa(x,s)$ is a Carath\'eodory function, and assume that 
	there are constants $0 <\alpha < \beta$ with
	\begin{align}\label{eqn:A1:elliptic}
	\alpha (t-s) \le \kappa(x, t) t - \kappa(x, s)s \le \beta(t-s), \quad 0\le s\le t
	\end{align}
	for $\ale x \in \Omega$.
\end{assumption}
\noindent Note that this assumption implies that $ \alpha \le \kappa(x,t) \le \beta$ $\ale x \in \Omega$ and for all $t>0$.
For simplicity, we denote $\kappa(s):=\kappa(x, s)$ in the rest of paper.

The analysis of standard conforming finite element method for general  quasilinear problems was discussed in \cite{Feistauer.M;Zenisek.A1987}.  For the class of monotone quasilinear PDEs \eqref{eqn:model}, the finite element error estimates were developed in \cite{Chow.S1989} for the differential operators defined on a reflexive Banach space. In \cite{HoustonRobsonSuli2005}, a one-parameter family of $hp$-version discontinuous Galerkin finite element methods was developed and analyzed for the numerical approximation of this type of quasilinear elliptic equations, subject to mixed Dirichlet-Neumann boundary conditions on $\partial \Omega$.  The finite volume methods were discussed in \cite{BiLin2012}.

First proposed in \cite{WangYe2013}, the weak Galerkin (WG) finite element methods are based on the novel idea of using \emph{weak functions} and their \emph{weak derivative} (see Section~\ref{sec:wg} for more details) in the design of numerical approximation schemes. Due to the discontinuous nature of the methods, the WG methods are very flexible in solving a variety of PDEs on general polytopal meshes. For the second order elliptic quasilinear PDEs, the existence of solutions of the WG methods was shown in \cite{MuWangYe2015} by a Schauder fixed point argument. However, the uniqueness and the error estimations of the numerical approximations are restricted only to the linear PDEs, and have not been addressed for the nonlinear ones.  Only recently in \cite{SunHuangWang2018}, the authors gave the well-posedness and error estimate in the energy norm for the monotone quasilinear PDEs \eqref{eqn:model}.

One disadvantage of the aforementioned WG methods, as well as other classes of discontinuous finite element methods (e.g. discontinuous Galerkin methods), is the existence of stabilization terms, which are usually necessary to enforce weak continuity of the discontinuous solutions across element boundaries. Removing stabilizers from discontinuous finite element methods will simplify formulations and reduce programming complexity significantly. Motivated by this, a class of stabilizer-free WG methods were first proposed and analyzed recently in \cite{YeZhang2019a} for Poisson's equation. In this new formulation, the WG methods can be viewed as the counterpart of the weak formulation of the continuous problem by replacing the classical gradient by the weak gradient operator. 

The goal of this paper is to formulate and analyze this stabilizer-free WG methods for the monotone quasilinear PDEs \eqref{eqn:model}. With the structural assumption~\ref{A1:elliptic}, we show the stabilizer-free WG formulation satisfies certain continuity and monotonicity properties. These properties imply the discrete problem has a unique solution, thanks to a nonlinear version of the Lax-Milgram theorem (Theorem~\ref{thm:fixed}) for monotone operators. We then derive optimal error estimates in the energy norm.

The rest of this paper is organized as follows. In Section~\ref{sec:wg}, we introduce basic notation, and present the stabilizer-free WG methods for the model problem \eqref{eqn:model}. In Section~\ref{sec:wellposedness}, we discuss the existence and uniqueness of the discrete problem. We first present an abstract existence and uniqueness of nonlinear operator equation, then verify the stabilizer-free WG methods satisfies the conditions based on the Assumption~\ref{A1:elliptic}. In Section~\ref{sec:error}, we show the main error estimate in the energy norm. In Section~\ref{sec:num}, we present some numerical experiments to confirm the theory. The paper ends with some concluding remarks and prospects for future work.

\section{Weak Galerkin Method}
\label{sec:wg}
For any given subset $D\subseteq\Omega$, we use the standard
definition of Sobolev spaces $H^s(D)$ with $s\ge 0$. The associated inner product,
norm, and semi-norms in $H^s(D)$ are denoted by
$(\cdot,\cdot)_{s,D}$, $\|\cdot\|_{s,D}$, and $|\cdot|_{s,D}$, respectively. When $s=0$, $H^0(D)$ coincides with the space
of square integrable functions $L^2(D)$. In this case, the subscript
$s$ is suppressed from the notation of norm, semi-norm, and inner
products. Furthermore, the subscript $D$ is also suppressed when
$D=\Omega$. Throughout the paper, we use $C$ to denote a generic positive constant that is independent of the meshsize and the solutions, and may take different values in different appearance.

Let $\cT_h$ be a partition of the domain $\Omega$ consisting of
polygons in two dimensions or polyhedra in three dimensions satisfying a set of shape-regular conditions (see \cite{MuWangYe2015} for example).
For every element $T\in \cT_h$, we
denote by $h_T$ the diameter of $T$ and mesh size $h:=\max_{T\in\cT_h} h_T$
for $\cT_h$.

We introduce the \emph{weak function} $v = \{v_0, v_b\}$ that allows $v$ to take different forms in the interior and on the boundary of each element $T\in \cT_{h}$:
$$
v=
\left\{
\begin{array}{l}
\displaystyle
v_0,\quad {\rm in}\; T,
\\ [0.08in]
\displaystyle
v_b,\quad {\rm on}\;\partial T.
\end{array}
\right.
$$
Given an integer $k\ge 1$, we define a local finite element space $V_h(T)$ on each element $T\in \cT_{h}$ as follows
\begin{eqnarray}\label{Vh}
V_h(T)=\{v=\{v_0,v_b\}: v_0 \in \P_k(T), \; v_b|_e\in \P_{k}(e), e\in\partial T\}.
\end{eqnarray}
A global finite element space $V_h$ is then derived by patching all the local elements $V_h(T)$ with  common values on interior edges. Let $V_h^0$ be a subspace of $V_h$ consisting of functions with vanishing boundary.

For any $v=\{v_0,v_b\}$, the \emph{discrete weak gradient} $\nabla_{w} v\in [\P_{j}(T)]^d$ is defined as the unique vector field satisfying
\begin{equation}\label{eqn:dwg}
(\nabla_{w} v, \ \bs \tau)_T = -( v_0, \ \nabla\cdot \bs \tau)_T + \l v_b, \bs \tau\cdot \bs n\r_{\partial T},\quad
\forall \bs \tau\in [\P_{j}(T)]^d,
\end{equation}
where $j>k$ is an integer to be specified later (see Lemma~\ref{lm:equiv}).
For simplicity, we adopt the following notations,

Then the WG scheme for \eqref{eqn:model} is to find $u_{h} = \{u_{0}, u_{b}\} \in V_{h}^{0}$ such that:
\begin{equation}\label{eqn:wg}
a_h(u_h;u_h, v):=\left(\kappa(|\nabla_w u_h|) \nabla_w u_h,\nabla_w v \right)_{\cT_h}=(f,\; v_0) \quad\forall v=\{v_0,v_b\}\in V_h^0.
\end{equation}

Let $Q_0$, $Q_b$ and $\Q_h$ be the locally defined $L^2$ projections onto $\P_{k}(T)$, $\P_{k}(e)$ and $[\P_{j}(T)]^d$ accordingly on each element $T\in\cT_h$ and $e\subset \partial T$.
For the exact solution $u$ of \eqref{eqn:model}, we define $Q_h u$ as
$$
Q_h u = \{Q_0 u,Q_bu\}\in V_h.
$$

For any $v\in V_h+H^1(\Omega)$, we introduce the following energy norm and the corresponding inner product: 
\begin{equation}\label{3barnorm}
\tbar v\tbar^2=(\nabla_w v,\nabla_w v)_{\cT_h}.
\end{equation}
We also define a discrete $H^1$ semi-norm as follows:
\begin{equation}\label{norm}
\|v\|_{1,h} = \left( \sum_{T\in\cT_h}\left(\|\nabla
v_0\|_T^2+h_T^{-1} \|  v_0-v_b\|^2_{\partial T}\right) \right)^{\frac{1}{2}}.
\end{equation}
It is easy to see that $\|v\|_{1,h}$ defines a norm in $V_h^0$. The following lemma indicates that $\|\cdot\|_{1,h}$ is equivalent
to the $\tbar\cdot\tbar$ in \eqref{3barnorm}.

\begin{lemma}[\cite{YeZhang2019a,YeZhang2020}]
	\label{lm:equiv} 
	Let $j=n+k-1$, where $n$ is the number of edges (faces) in each element. There exist two positive constants $C_1$ and $C_2$ such
	that for any $v=\{v_0,v_b\}\in V_h$, we have
	\begin{equation}\label{eqn:equiv}
	C_1 \|v\|_{1,h}\le \tbar v\tbar \leq C_2 \|v\|_{1,h}.
	\end{equation}
\end{lemma}
We remark that even though in Lemma~\ref{lm:equiv} we required $j = n+k-1$, our numerical experiments in Section~\ref{sec:num} indicate that we can still get optimal error estimates with $j=k+1$ or $j=k+2$. 
\section{Existence and Uniqueness}
\label{sec:wellposedness}
In this section, we show the problem \eqref{eqn:wg} has a unique solution. For this purpose, we first present an abstract theorem. Let $H$ be a Hilbert space with the inner product denoted by $(\cdot, \cdot)_H$ and the induced norm $\|\cdot \|_H$.  We say a (nonlinear) operator $N: H\to H$  is  \emph{strongly monotone} if there exists a constant $\lambda>0$ such that
\begin{equation}
\label{eqn:monotoneOp}
(N(u) - N(v), u-v)_H \ge \lambda \|u-v\|_H^2;
\end{equation}
$N: H\to H$  is \emph{Lipschitz continuous} if there is a constant $\Lambda > 0$ such that
\begin{equation}
\label{eqn:lipschitzOp}
\|N(u) - N(v)\|_H \le \Lambda \|u - v\|_H.
\end{equation}
The following theorem ({cf. \cite{Necas.J1986}}) can be viewed as the
nonlinear version of the Lax-Milgram theorem. For completeness, we include a simple proof here. 
\begin{theorem}
	\label{thm:fixed}
	Let the operator $N: H\to H$ be strongly monotone \eqref{eqn:monotoneOp} and Lipschitz continuous \eqref{eqn:lipschitzOp}. Then $N(u) =f $ has a unique solution for all $f\in H$.
\end{theorem}
\begin{proof}
	Let $A u = u - \varepsilon (N(u) - f)$. It is clear that the solution to the equation $N(u) = f$ is equivalent to the fixed point $Au =u$ of $A$.  By the strong monotonicity \eqref{eqn:monotoneOp} and Lipschitz continuity \eqref{eqn:lipschitzOp}, the operator $A: H \to H$ satisfies
	\begin{align*}
	\|A u - Av\|_H^2 &=\| u- v\|_H^2 + \varepsilon^2 \| N(u) - N(v)\|_H^2 - 2\varepsilon (N(u) - N(v), u - v)_H\\
	&\le\left(1- 2\varepsilon \lambda + \Lambda \varepsilon^2\right) \|u - v\|_H^2.
	\end{align*}
	Clearly, for any $\varepsilon \in (0, 2\lambda/\Lambda^2)$, $A: H\to H$ is a contraction mapping. By Banach fixed-point theorem, $A$ has a unique fixed point. Hence $N(u) = f$ has a unique solution.
\end{proof}

\begin{remark}
	\label{rk:fixedpoint}
	By the proof of Theorem~\ref{thm:fixed}, we can construct a fixed point iteration $u_{n+1} = A u_n$ for any initial guess $u_0\in H$. For appropriate choice of $\varepsilon$, this iteration is guaranteed to converge (globally). This is in fact the relaxed Picard iteration, which is the algorithm used in Section~\ref{sec:num} for solving the nonlinear problems.
\end{remark}

Based on Theorem~\ref{thm:fixed}, in order to show \eqref{eqn:wg} has a unique solution, we just need to verify the related discrete nonlinear operator satisfies the strong monotonicity \eqref{eqn:monotoneOp} and the Lipschitz continuity \eqref{eqn:lipschitzOp}. These properties can be obtained by the Assumption \ref{A1:elliptic} on the coefficient $\kappa$.  We first show the following continuity and monotonicity lemma.
\begin{lemma}
	\label{lm:coef}
	If the coefficient $\kappa$ satisfies the Assumption \ref{A1:elliptic}, then we have
	\begin{align}
	&\alpha |\bs \xi - \bs \eta|^2 \le \left(\kappa( |\bs \xi|) \bs \xi - \kappa( |\bs \eta|) \bs \eta, \bs \xi - \bs \eta\right),  \qquad \forall \bs \xi, \bs  \eta \in \R^d \label{eqn:mono}\\
	&\left|\kappa( |\bs \xi|) \bs \xi - \kappa( |\bs \eta|) \bs \eta\right| \le \beta |\bs \xi - \bs \eta|, \qquad \forall \bs \xi, \bs  \eta \in \R^d.  \label{eqn:cont}
	\end{align}
\end{lemma}

\begin{proof}
	To prove \eqref{eqn:mono}, we use the lower bound in the Assumption~\ref{A1:elliptic}.
	\begin{align*}
	&\left(\kappa( |\bs \xi|) \bs \xi - \kappa( |\bs \eta|) \bs \eta, \bs \xi - \bs \eta\right) = \kappa( |\bs \xi|) |\bs \xi|^2 +  \kappa( |\bs \eta|) |\bs \eta|^2 - [\kappa( |\bs \xi|) + \kappa( |\bs \eta|)] \bs \xi \cdot \bs \eta\\
	&= \left(\kappa( |\bs \xi|) |\bs \xi| -  \kappa( |\bs \eta|) |\bs \eta|\right) (|\bs \xi| - |\bs \eta|)  + [\kappa( |\bs \xi| )+ \kappa( |\bs \eta|)] \left( |\bs \xi| |\bs \eta| - \bs \xi \cdot \bs \eta\right)\\
	&\ge \alpha (|\bs \xi| - |\bs \eta|)^2 + 2\alpha \left( |\bs \xi| |\bs \eta| - \bs \xi \cdot \bs \eta\right)\\
	&= \alpha | \bs \xi -\bs \eta|^2.
	\end{align*}
	
	On the other hand, by the upper bound in the Assumption~\ref{A1:elliptic}, we have,
	\begin{align*}
	&\left|\kappa( |\bs \xi|) \bs \xi - \kappa( |\bs \eta|) \bs \eta\right|^2  = \kappa( |\bs \xi|) |\bs \xi|^2 + \kappa( |\bs \eta|) |\bs \eta|^2  - 2\kappa( |\bs \xi|) \kappa( |\bs \eta|)\bs \xi \cdot\bs \eta \\
	&=\left(\kappa( |\bs \xi|) |\bs \xi| - \kappa( |\bs \eta|) |\bs \eta|\right)^2 + 2\kappa( |\bs \xi|) \kappa( |\bs \eta|) \left(|\bs \xi| |\bs \eta| -\bs \xi \cdot\bs \eta \right)\\
	&\le \beta^2 \left(|\bs \xi|  - |\bs \eta|\right)^2 + 2\beta^2 \left(|\bs \xi| |\bs \eta| -\bs \xi \cdot\bs \eta \right)\\
	&= \beta^2 |\bs \xi  -\bs \eta|^2.
	\end{align*}
	Taking square root on both sides, we  obtain  \eqref{eqn:cont}.
\end{proof}

The estimate \eqref{eqn:mono} in Lemma~\ref{lm:coef} implies the following strong monotonicity of $a_h$.
\begin{lemma}
	\label{lm:mon}
	If the coefficient $\kappa$ satisfies the Assumption \ref{A1:elliptic}, then the nonlinear form $a_h$ defined in \eqref{eqn:wg} is strongly monotone in the sense that
	\[
	\alpha\tbar u_1 - u_2\tbar^2  \le a_h(u_1; u_1, u_1-u_2) - a_h(u_2; u_2, u_1 - u_2), \quad \forall u_1, u_2\in V_h^0.
	\]
\end{lemma}
\begin{proof}
	The conclusion is a direct consequence of the inequality \eqref{eqn:mono}.
\end{proof}
On the other hand,  the estimate \eqref{eqn:cont} implies  the following Lipschitz continuity of $a_h$.
\begin{lemma}
	\label{lm:lipschitzWG}
	If the coefficient $\kappa$ satisfies the Assumption \ref{A1:elliptic}, then the nonlinear form $a_h$ defined in \eqref{eqn:wg} is Lipschitz continuous in the sense that
	\[|a_h(u_1; u_1, v) - a_h(u_2; u_2, v)| \le \beta \tbar u_1 -u_2\tbar \tbar v\tbar, \qquad \forall u_1, u_2, v\in V_h^0.\]
\end{lemma}
\begin{proof}
	By inequality \eqref{eqn:cont}, we immediately get that
	\begin{align*}
	|a_h(u_1; u_1, v) - a_h(u_2; u_2, v)|  &= \left|\left(\kappa( |\nabla_{w} u_1|)\nabla_{w} u_1  - \kappa( |\nabla_{w} u_2|)\nabla_{w} u_2, \nabla_w v\right)_{\cT_{h}} \right| \\	
	&\le \beta \tbar u_1 - u_2 \tbar \tbar v\tbar.
	\end{align*}
	This completes the proof.
\end{proof}

Now, we are ready to present the existence and uniqueness of \eqref{eqn:wg}.
\begin{theorem}
	If the coefficient $\kappa$ satisfies the Assumption \ref{A1:elliptic}, then the weak Galerkin finite element scheme \eqref{eqn:wg} has a unique
	solution.
\end{theorem}

\begin{proof}
	In order to use the abstract existence and uniqueness result in Theorem \ref{thm:fixed}, we first need to rewrite the equation \eqref{eqn:wg} in the operator form on the finite dimensional Hilbert space $V_h^0$ with the inner product and norm defined in \eqref{3barnorm}. For any $v = \{v_0, v_b\} \in V_h^0$ we have (cf. \cite[Lemma 7.1]{MuWangYe2015}, and Lemma~\ref{lm:equiv} )
	\begin{equation}
	\|v_0\| \le C \| v\|_{1,h} \le C \tbar v\tbar .
	\end{equation}
	Hence, we have
	\[
	|(f,v_0) | \le \|f\| \|v_0\| \le C \|f\| \tbar v\tbar.
	\]
	This implies that $f$ is a bounded linear functional on $V_h^0$. Then by Riesz representation theorem, there exists an $\bs f\in V_h^0$ such that
	\[
	(f, v) = \langle\bs f, v\rangle, \qquad \forall v\in V_h^0.
	\]
	Here, $\langle \cdot, \cdot \rangle$ is the inner product defined in \eqref{3barnorm}.
	
	Fix a $w\in V_h^0$, consider the linear functional $\Phi_w(v) := a_h(w;w, v)$ for any $v\in V_h^0.$ By Assumption~\ref{A1:elliptic}, it is clear that
	\begin{align*}
	\Phi_w(v) &\le \beta \tbar w\tbar \tbar v\tbar.
	\end{align*}
	That is, $\Phi_w$ is a bounded linear functional on $V_h^0$. By Riesz representation theorem, there exists a $N(w) \in V_h^0$ such that
	\[
	\Phi_w (v) = \langle N(w), v \rangle, \qquad \forall v\in V_h^0.
	\]
	Therefore, \eqref{eqn:wg} is equivalent to the following operator form: Find $u_h\in V_h^0$ such that
	\begin{equation}
	\label{eqn:op}
	N(u_h) = \bs f.
	\end{equation}
	Lemma~\ref{lm:mon} implies that $N$ is strongly monotone, and Lemma~\ref{lm:lipschitzWG} implies that $N$ is Lipschitz continuous on $V_h^0$.  By Theorem~\ref{thm:fixed}, \eqref{eqn:op} has a unique solution. Therefore, \eqref{eqn:wg} has a unique solution.
\end{proof}

\section{Error Analysis}
\label{sec:error}

In this section, we establish the error estimate for the WG finite element approximation \eqref{eqn:wg} in the energy norm defined in \eqref{3barnorm}. For this purpose, we first introduce the following lemmas.
\begin{lemma}[{\cite[Lemma 2.1]{YeZhang2019a}}]
	\label{lm:commute}
	Let $v\in H^1(\Omega)$, then on any element $T\in \cT_h$, it holds
	\begin{equation}
	\label{eqn:commute}
	\nabla_w v = \Q_h \nabla v.
	\end{equation}
\end{lemma}
\begin{lemma}[{\cite[Lemma 4.3]{YeZhang2019a}}]
	\label{lm:wgh1}
	Let $w \in H^{k+1}(\Omega)$. Then
	\begin{equation}
	\label{eqn:wgh1}
	\tbar w- Q_h w\tbar \le C h^k |w|_{k+1}.
	\end{equation}
\end{lemma}

Now we are ready to prove the main theorem.
\begin{theorem}
	\label{thm:main}
	Let $u_h\in V_h$ be the weak Galerkin finite element solution to \eqref{eqn:wg}. Assume that $u\in H^{k+1}(\Omega)$ is the exact solution to \eqref{eqn:model}, and the coefficient $\kappa$ satisfies the Assumption \ref{A1:elliptic}. If in addition, $\kappa \in W^{k,\infty}(\Omega\times \R^+)$, then there exists a constant $C>0$ independent of $h$, $u$ and $u_h$ such that
	\[
	\tbar u - u_h\tbar \le C h^k \|u\|_{k+1}.
	\]
\end{theorem}
\begin{proof}
	By \eqref{eqn:wgh1} and the triangle inequality, it suffices to show that
	\begin{equation}
	\label{eqn:uh-Qu}
	\tbar u_h - Q_h u\tbar \le C h^k  |u|_{k+1}.
	\end{equation}
	For simplicity, let $e_h := u_h - Q_h u = \{e_0, e_b\}\in V_h^0$, where $e_0 = u_0 - Q_0 u$ and $e_b = u_b - Q_b u$. We also denote  $\bs \sigma(v):= \kappa( |\nabla v|) \nabla v$ and $\bs \sigma_w(v):= \kappa(|\nabla_w v|)\nabla_{w} v$.
	
	We test the continuous equation \eqref{eqn:model} with $v_{0}$ for any $v=\{v_{0}, v_{b}\}\in V_{h}^{0}$. Notice that the flux  $\bs \sigma(u)$ is continuous in the normal direction, we obtain
	\begin{equation}
	\label{eqn:continuous}
	(\bs \sigma(u), \nabla v_{0})_{\cT_{h}} - \langle \bs \sigma(u) \cdot \bs n, v_{0}-v_{b}\rangle_{\partial \cT_h} = (f, v_{0})_{\cT_{h}}.
	\end{equation}
	Then by the strong monotonicity Lemma~\ref{lm:mon}, we have
	\begin{align}
	\alpha\tbar e_h\tbar^2 &\le a_h(u_h; u_h, e_h) - a_h(Q_h u; Q_h u, e_h) \nonumber\\
	&=(f, e_0)_{\cT_{h}} - a_h(Q_h u; Q_h u, e_h) \nonumber\\
	&=(\bs \sigma(u), \nabla e_0)_{\cT_{h}} - \langle \bs \sigma(u) \cdot \bs n, e_0 - e_b\rangle_{\partial \cT_h} - (\bs \sigma_w(Q_h u), \nabla_w e_h)_{\cT_{h}}, \label{eqn:eh}
	\end{align}
	where we used \eqref{eqn:continuous} in the last step. 
	On each $T\in \cT_h$, it follows from integration by parts and the definition of the discrete weak gradient \eqref{eqn:dwg},
	\begin{align*}
	(\bs \sigma(u), \nabla e_0)_T &= (\Q_h (\bs \sigma(u)), \nabla e_0)_T\\
	&= - (\nabla \cdot \Q_h(\bs \sigma (u)), e_0)_T + \langle \Q_h(\bs \sigma(u))\cdot \bs n, e_0\rangle_{\partial T}\\
	& =(\Q_h (\bs \sigma(u)), \nabla_w e_h)_T + \langle \Q_h(\bs \sigma(u))\cdot \bs n, e_0-e_b\rangle_{\partial T}.
	\end{align*}
	Therefore, we get
	\begin{align}
	(\bs \sigma(u), \nabla e_0)_{\cT_h} =(\Q_h (\bs \sigma(u)), \nabla_w e_h)_{\cT_{h}} + \langle \Q_h(\bs \sigma(u))\cdot \bs n, e_0-e_b\rangle_{\partial \cT_h}. \label{eqn:sigmaue0}
	\end{align}
	Replacing the first term on the right-hand side of the inequality \eqref{eqn:eh} with the right-hand side of \eqref{eqn:sigmaue0},  we obtain
	\begin{align}
	\alpha\tbar e_h\tbar^2 &\le (\Q_h(\bs \sigma(u)) - \bs \sigma_w(Q_h u), \nabla_w e_h)_{\cT_{h}} \nonumber\\
	&-\langle \left(\Q_h(\bs \sigma(u)) - \bs \sigma(u)\right)\cdot \bs n, e_0-e_b\rangle_{\partial \cT_h}. \label{eqn:m1}
	\end{align}
	
	To estimate the first term in \eqref{eqn:m1}, we have
	\begin{align}
	\|\Q_h(\bs \sigma(u)) - \bs \sigma_w(Q_h u)\| &\le \|(I-\Q_h)\bs \sigma(u)\| + \|\bs \sigma(u) -\bs \sigma_w(Q_h u)\|  \nonumber\\
	&\le C h^k |\bs \sigma(u)|_{k} + \beta {\|\nabla u- \nabla_w (Q_h u)\|}	\nonumber\\
	&\le Ch^k |u|_{k+1} +  \beta \left( \|\nabla u- \nabla_w u\|+\|\nabla_w  (u-Q_h u)\| \right) \nonumber\\
	&=Ch^k |u|_{k+1} +  \beta \left( \|\nabla u- \Q_h\nabla u\|+\tbar u-Q_h u\tbar \right) \nonumber\\
	&\le Ch^k |u|_{k+1} \label{eqn:m2}
	\end{align}
	where in the second inequality we used the inequality \eqref{eqn:cont}, in the third inequality we used the condition $\kappa\in W^{k,\infty}(\Omega\times \R^+)$,  and in the fourth equality we used Lemma~\ref{lm:commute} for the second term and the definition of the norm \eqref{3barnorm} for the last term.
	Here, we have also used the approximation properties for $\Q_h$ and $Q_h$ in the last step. In particular, we have 
		\begin{equation}
		\label{eqn:Qh}
		\|\bs q - \Q_h \bs q\| \le Ch^k |\bs q|_{k}
		\end{equation}
	(see \cite[Lemma 4.1]{WangYe2014} for a proof of this inequality on general polytopal mesh); for the last term  $\tbar u-Q_h u\tbar$, the estimate follows directly from \eqref{eqn:wgh1}.
	Therefore, we obtain the following estimate for the first term in~\eqref{eqn:m1}
	\begin{align}
	(\Q_h(\bs \sigma(u)) - \bs \sigma_w(Q_h u), \nabla_w e_h)_{\cT_{h}} &\le \|\Q_h(\bs \sigma(u)) - \bs \sigma_w(Q_h u)\| \tbar e_h \tbar \nonumber\\
	&\le C h^k |u|_{k+1} \tbar e_h \tbar. \label{eqn:term1}
	\end{align}
	
	Now we turn to estimate the second term in \eqref{eqn:m1}. We need the following trace inequality for any function $v\in H^1(T)$ (cf. \cite{WangYe2014}):
	\begin{equation}
	\label{eqn:trace}
	\|v\|_{\partial T}^2 \le C(h_T^{-1} \|v\|_T^2 + h_T \|\nabla v\|_T^2).
	\end{equation}
	By this trace inequality and the estimate \eqref{eqn:m2}, we have the following estimate for the second term in \eqref{eqn:m1}:
		\begin{align}
		&\langle \left(\Q_h(\bs \sigma(u)) - \bs \sigma(u)\right)\cdot \bs n, e_0-e_b\rangle_{\partial \cT_h} \nonumber \\
		& \le C \sum_{T\in \cT_h} \|\Q_h(\bs \sigma(u)) - \bs \sigma( u)\|_{\partial T} \| e_0 - e_b\|_{\partial T} \nonumber\\
		&\le C \left( \sum_{T\in \cT_h} h_T\|\Q_h(\bs \sigma(u)) - \bs \sigma(u)\|^2_{\partial T}\right)^{1/2} \left( \sum_{T\in \cT_h}h_T^{-1}\| e_0 - e_b\|^2_{\partial T}\right)^{1/2} \nonumber\\
		&\le C h^k |\bs \sigma(u)|_k \tbar e_h\tbar 
		\le C h^k |u|_{k+1} \tbar e_h\tbar. \label{eqn:term2}
		\end{align}
	In the last inequality, we used the trace inequality \eqref{eqn:trace}, the estimate \eqref{eqn:Qh}, the condition $\kappa\in W^{k,\infty}(\Omega\times \R^+)$, and the norm equivalency \eqref{eqn:equiv}. The conclusion then follows directly from inequalities \eqref{eqn:term1} and \eqref{eqn:term2}. This completes the proof.
\end{proof}

\section{Numerical Experiments}
\label{sec:num}
We apply the new stabilizer-free weak Galerkin finite element method with various polynomial degrees and  on various polygonal grids, to two monotone elliptic equations. Even though we did not give the analysis for the $L^2$ error estimates, we present them in these numerical examples for comparison. Note the $L^2$ error between the exact solution $u\in H_0^1(\Omega)$ and the WG approximation $u_h=\{u_0, u_b\} \in V_h^0$ is defined by 
\[
\|u - u_h\|_0^2:= \sum_{T\in \cT_h}  \int_T |u - u_0|^2 dx. 
\]

\subsection{Example 1.}  We solve problem \eqref{eqn:model} on square domain $\Omega=(0,1)^2$,  where the
coefficient function and the exact solution are
\begin{align} \label{s1} \kappa(|\nabla u|) = 1+e^{-|\nabla u|^2}, \quad
u=\sin(\pi x)(y-y^2).
\end{align}
This function $\kappa$ satisfies condition \eqref{eqn:A1:elliptic} with $\alpha=1-\sqrt{2/e}$ and
$\beta=2$.

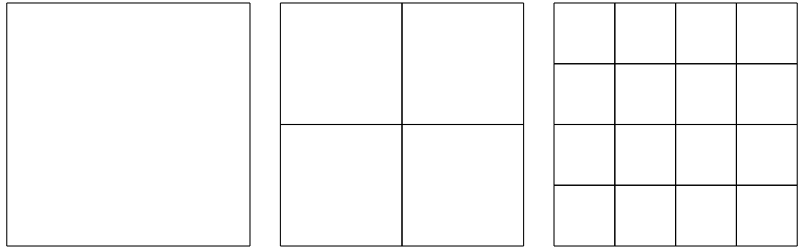
\begin{figure}[h!]
	\begin{center} \setlength\unitlength{1.15pt}
		\begin{picture}(260,80)(0,0)
		\def\tr{\begin{picture}(20,20)(0,0)\put(0,0){\line(1,0){20}}\put(0,20){\line(1,0){20}}
			\put(0,0){\line(0,1){20}} \put(20,0){\line(0,1){20}}
			\end{picture}}
		
		{\setlength\unitlength{4.6pt}
			\multiput(0,0)(20,0){1}{\multiput(0,0)(0,20){1}{\tr}}}
		
		{\setlength\unitlength{2.3pt}
			\multiput(45,0)(20,0){2}{\multiput(0,0)(0,20){2}{\tr}}}
		
		\multiput(180,0)(20,0){4}{\multiput(0,0)(0,20){4}{\tr}}
		
\end{picture}\end{center}
\caption{The level one, level two and level three rectangular grids. }\label{grid1}
\end{figure}

We compute the solution \eqref{s1} on two types of grids, shown in Fig. \ref{grid1}
and Fig. \ref{grid2}.
We use $\P_k$ ($k=1,2,3,4$ in \eqref{Vh}) weak Galerkin finite elements with $\P_{k+1}$
weak gradient ($j=k+1$) in \eqref{eqn:dwg}) on rectangular grids (Fig. \ref{grid1}),
and $\P_k$ ($k=1,2,3$ in \eqref{Vh}) weak Galerkin finite elements with  $\P_{k+2}$
weak gradient ($j=k+2$) in \eqref{eqn:dwg})
on  polygonal grids (Fig. \ref{grid2}).
In the computation,  the function $\kappa(|\nabla u_h|)$ is interpolated in to the
discontinuous $\P_{k-1}$ space on the same grid.
On each level,  we solve the nonlinear discrete equations by the relaxed-Picard iteration.
The errors and the order of convergence are listed in Tables \ref{t1}--\ref{t2} for the  computation on two types of grids,  respectively.

\begin{figure}[h!]
\begin{center} \setlength\unitlength{1.15pt}
	\begin{picture}(180,80)(0,0) 
	\def\gon{\begin{picture}(20,20)(0,0)\put(0,0){\line(1,0){20}}\put(0,20){\line(1,0){20}}
		\put(0,0){\line(0,1){20}} \put(20,0){\line(0,1){20}}
		\put(0,0){\line(3,2){7.2}}\put(20,0){\line(-3,2){7.2}}\put(7.2,4.8){\line(1,0){5.6}}
		\put(7.2,4.8){\line(-1,2){2.6}}\put(12.8,4.8){\line(1,2){2.6}}
		\put(0,20){\line(3,-2){7.2}}\put(20,20){\line(-3,-2){7.2}}\put(7.2,15.2){\line(1,0){5.6}}
		\put(7.2,15.2){\line(-1,-2){2.6}}\put(12.8,15.2){\line(1,-2){2.6}}
		\end{picture}} 
	{\setlength\unitlength{4.6pt}
		\multiput(0,0)(20,0){1}{\multiput(0,0)(0,20){1}{\gon}}}
	{\setlength\unitlength{2.3pt}
		\multiput(50,0)(20,0){2}{\multiput(0,0)(0,20){2}{\gon}}}

\end{picture}\end{center}
\caption{The level one and level two quadrilateral-pentagonal-hexagonal grids. }\label{grid2}
\end{figure}
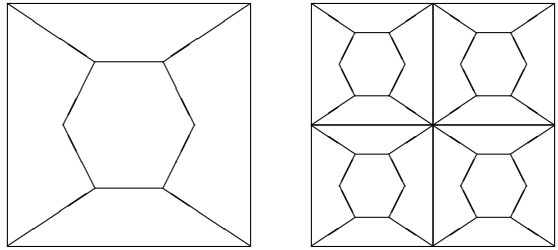

\begin{table}[!ht]
\centering \renewcommand{\arraystretch}{1.1}
\caption{The error and the order of convergence for \eqref{s1} by $\P_k$
WG finite elements \eqref{Vh} on square grids (Fig. \ref{grid1}) }
\label{t1}
\begin{tabular}{c|cc|cc}
\hline
level & $\|u_h-  u\|_0 $  &rate & $ \tbar u_h- u\tbar $ &rate    \\
\hline
&\multicolumn{4}{c}{$\P_1$ element with $\P_2$ gradient ($j=2$ in \eqref{eqn:dwg})}
\\ \hline
4&   0.6829E-03 & 2.58&   0.3218E-01 & 1.97 \\
5&   0.1343E-03 & 2.35&   0.8080E-02 & 1.99 \\
6&   0.3071E-04 & 2.13&   0.2022E-02 & 2.00 \\
7&   0.7483E-05 & 2.04&   0.5056E-03 & 2.00 \\
\hline
&\multicolumn{4}{c}{$\P_2$ element with $\P_3$ gradient ($j=3$ in \eqref{eqn:dwg})}
\\ \hline
4&   0.7518E-04 & 3.93&   0.5291E-02 & 2.97 \\
5&   0.5663E-05 & 3.73&   0.6690E-03 & 2.98 \\
6&   0.4687E-06 & 3.59&   0.8390E-04 & 3.00 \\
7&   0.4698E-07 & 3.32&   0.1049E-04 & 3.00 \\
\hline
&\multicolumn{4}{c}{$\P_3$ element with $\P_4$ gradient ($j=4$ in \eqref{eqn:dwg})}
\\ \hline
3&   0.3009E-03 & 6.82&   0.9279E-02 & 4.72 \\
4&   0.1304E-04 & 4.53&   0.5598E-03 & 4.05 \\
5&   0.7081E-06 & 4.20&   0.3455E-04 & 4.02 \\
6&   0.3753E-07 & 4.24&   0.2298E-05 & 3.91 \\
\hline
&\multicolumn{4}{c}{$\P_4$ element with $\P_5$ gradient ($j=5$ in \eqref{eqn:dwg})}
\\ \hline
2&   0.3124E-01 & 2.67&   0.2036E+00 & 3.04 \\
3&   0.1762E-03 & 7.47&   0.2652E-02 & 6.26 \\
4&   0.1811E-05 & 6.60&   0.5513E-04 & 5.59 \\
5&   0.6090E-07 & 4.89&   0.3395E-05 & 4.02 \\
\hline
\end{tabular}%
\end{table}%

\begin{table}[!ht]
\centering \renewcommand{\arraystretch}{1.1}
\caption{The error and the order of convergence for \eqref{s1} by $\P_k$
WG finite elements \eqref{Vh} on quadrilateral-pentagonal-hexagonal  grids (Fig. \ref{grid2}) }
\label{t2}
\begin{tabular}{c|cc|cc}
\hline
level & $\|u_h-  u\|_0 $  &rate & $ \tbar u_h- u\tbar $ &rate    \\
\hline
&\multicolumn{4}{c}{$\P_1$ element with $\P_3$ gradient ($j=3$ in \eqref{eqn:dwg})}
\\ \hline
3&   0.3785E-02 & 0.92&   0.1751E+00 & 0.94 \\
4&   0.1315E-02 & 1.53&   0.8855E-01 & 0.98 \\
5&   0.3692E-03 & 1.83&   0.4418E-01 & 1.00 \\
6&   0.9499E-04 & 1.96&   0.2206E-01 & 1.00 \\
\hline
&\multicolumn{4}{c}{$\P_2$ element with $\P_4$ gradient ($j=4$ in \eqref{eqn:dwg})}
\\ \hline
3&   0.1219E-02 & 1.81&   0.1617E-01 & 1.86 \\
4&   0.3427E-03 & 1.83&   0.4533E-02 & 1.84 \\
5&   0.8914E-04 & 1.94&   0.1187E-02 & 1.93 \\
6&   0.2254E-04 & 1.98&   0.3011E-03 & 1.98 \\
\hline
&\multicolumn{4}{c}{$\P_3$ element with $\P_5$ gradient ($j=5$ in \eqref{eqn:dwg})}
\\ \hline
2&   0.1024E-02 & 2.75&   0.1865E-01 & 1.99 \\
3&   0.1550E-03 & 2.72&   0.4905E-02 & 1.93 \\
4&   0.1422E-04 & 3.45&   0.6854E-03 & 2.84 \\
5&   0.1489E-05 & 3.26&   0.8820E-04 & 2.96 \\
\hline
\end{tabular}%
\end{table}%

For a comparison,  we also solve this problem by the traditional weak Galerkin finite element
method,  i.e., the method with a stabilization/penalty.  
The WG with a stabilization  is to find $u_{h} = \{u_{0}, u_{b}\} \in V_{h}^{0}$ such that:
\begin{equation}\label{eqn:wg1}
\left(\kappa(|\nabla_w u_h|) \nabla_w u_h,\nabla_w v \right)_{\cT_h}
+  \langle \frac 1h (u_0-u_b), v_0-v_b\rangle_{\partial \mathcal T_h}=(f,\; v_0) 
\end{equation} for all $v=\{v_0,v_b\}\in V_h^0$,  where
\begin{equation}\label{stab1}
\langle \frac 1h (u_0-u_b), v_0-v_b\rangle_{\partial T_h}=
\sum_{T\in\mathcal T_h} \int_{e\subset\partial T} \frac 1h (u_0-u_b)(v_0-v_b) ds.
\end{equation} 
The computation is performed on a PC with an Intel i5-7200U CPU at 2.50GHz.
We list the total CPU time for the two methods in Table \ref{t1-2}.
The differences are negligible.
But the new method is slightly more accurate than the traditional weak Galerkin
finite element.
This is because the traditional WG has a penalty term which makes the solution more continuous
(both new and the traditional solutions are discontinuous)
but less flexible.
Of course, the $P_2$ solutions are much more accurate than the $P_1$ solutions with less computational time.

\begin{table}[!ht]
\centering 
\caption{A comparison of new WG \eqref{eqn:wg} and the traditional WG
\eqref{eqn:wg1}, for solving \eqref{s1} on square grids (Fig. \ref{grid1}) }
\label{t1-2}
\begin{tabular}{c|cc|cc}
\hline
level & $\|u_h-  u\|_0 $  &rate & $ \tbar u_h- u\tbar $ &rate    \\
\hline
&\multicolumn{4}{c}{New $\P_1$ WG by \eqref{eqn:wg}}
\\  
5&   0.1343E-03 & 2.35&   0.8080E-02 & 1.99 \\
6&   0.3071E-04 & 2.13&   0.2022E-02 & 2.00 \\
7&   0.7483E-05 & 2.04&   0.5056E-03 & 2.00 \\
&\multicolumn{4}{c}{ Total CPU time =   876.28125 seconds }\\
\hline
&\multicolumn{4}{c}{The traditional $\P_1$ WG by \eqref{eqn:wg1}}
\\
5&   0.1489E-03 & 2.28&   0.9474E-02 & 1.89 \\
6&   0.3520E-04 & 2.08&   0.2932E-02 & 1.69 \\
7&   0.8681E-05 & 2.02&   0.1132E-02 & 1.37 \\
&\multicolumn{4}{c}{ Total CPU time =    876.09375 seconds }\\
\hline
&\multicolumn{4}{c}{New $\P_2$ WG by \eqref{eqn:wg}}
\\  
4&   0.7447E-04 & 3.94&   0.5288E-02 & 2.97 \\
5&   0.5553E-05 & 3.75&   0.6681E-03 & 2.98 \\
6&   0.4609E-06 & 3.59&   0.8377E-04 & 3.00 \\ 
&\multicolumn{4}{c}{ Total CPU time =  371.625 seconds }\\
\hline
&\multicolumn{4}{c}{The traditional $\P_2$ WG by \eqref{eqn:wg1}}
\\
4&   0.7005E-04 & 4.00&   0.5564E-02 & 2.96 \\
5&   0.5040E-05 & 3.80&   0.7260E-03 & 2.94 \\
6&   0.4027E-06 & 3.65&   0.1024E-03 & 2.83 \\
&\multicolumn{4}{c}{ Total CPU time =  372.25 seconds }\\ \hline
\end{tabular}%
\end{table}%

\subsection{Example 2.}  We solve problem \eqref{eqn:model} on square domain $\Omega=(0,1)^2$ again,
where the
coefficient function and the exact solution are
\begin{align} \label{s2} \kappa(|\nabla u|) = \frac{3+2|\nabla u|}{1+|\nabla u|}, \quad
u=(x-x^2)\sin(\pi y).
\end{align}
This function $\kappa$ satisfies condition \eqref{eqn:A1:elliptic} with $\alpha=2$ and
$\beta=3$.
But this non-linear function has an additional singularity, comparing to that in Example 1,
the square-root singularity. 
This is related to the assumption $\kappa \in W^{k,\infty}(\Omega\times \R^+)$ in Theorem~\ref{thm:main}.  Here the singularity is caused by the function 
$|\nabla u|=\sqrt{ u_x^2+u_y^2}$, whose partial derivatives will blowup as $\nabla u \approx 0$ especially when the order $k>1$. 
Due to the loss of regularity of the coefficient $\kappa$, the higher order finite element methods do not perform as good as the  results in Example 1.

\begin{table}[h!]
\centering \renewcommand{\arraystretch}{1.1}
\caption{The error and the order of convergence for \eqref{s2} by $\P_k$
WG finite elements \eqref{Vh} on square grids (Fig. \ref{grid1}) }
\label{t3}
\begin{tabular}{c|cc|cc}
\hline
level & $\|u_h-  u\|_0 $  &rate & $ \tbar u_h- u\tbar $ &rate    \\
\hline
&\multicolumn{4}{c}{$\P_1$ element with $\P_2$ gradient ($j=2$ in \eqref{eqn:dwg})}
\\ \hline
4&   0.6126E-03 & 3.42&   0.3214E-01 & 2.06 \\
5&   0.7175E-04 & 3.09&   0.7950E-02 & 2.02 \\
6&   0.1003E-04 & 2.84&   0.1982E-02 & 2.00 \\
7&   0.1721E-05 & 2.54&   0.4952E-03 & 2.00 \\
\hline
&\multicolumn{4}{c}{$\P_2$ element with $\P_3$ gradient ($j=3$ in \eqref{eqn:dwg})}
\\ \hline
4&   0.3328E-03 & 3.95&   0.6875E-02 & 3.00 \\
5&   0.2511E-04 & 3.73&   0.8669E-03 & 2.99 \\
6&   0.2382E-05 & 3.40&   0.1095E-03 & 2.98 \\
7&   0.2701E-06 & 3.14&   0.1382E-04 & 2.99 \\
\hline
&\multicolumn{4}{c}{$\P_3$ element with $\P_4$ gradient ($j=4$ in \eqref{eqn:dwg})}
\\ \hline
2&   0.7066E-01 & 0.97&   0.3533E+00 & 1.82 \\
3&   0.4185E-02 & 4.08&   0.3976E-01 & 3.15 \\
4&   0.2706E-03 & 3.95&   0.4931E-02 & 3.01 \\
5&   0.2536E-04 & 3.42&   0.6241E-03 & 2.98 \\
\hline
&\multicolumn{4}{c}{$\P_4$ element with $\P_5$ gradient ($j=5$ in \eqref{eqn:dwg})}
\\ \hline
2&   0.7508E-01 & 0.98&   0.3588E+00 & 1.95 \\
3&   0.7335E-02 & 3.36&   0.4763E-01 & 2.91 \\
4&   0.6863E-03 & 3.42&   0.6028E-02 & 2.98 \\
5&   0.7295E-04 & 3.23&   0.7777E-03 & 2.95 \\
\hline
\end{tabular}%
\end{table}%

\begin{table}[h!]
\centering \renewcommand{\arraystretch}{1.1}
\caption{The error and the order of convergence for \eqref{s2} by $\P_k$
WG finite elements \eqref{Vh} on quadrilateral-pentagonal-hexagonal  grids (Fig. \ref{grid2}) }
\label{t4}
\begin{tabular}{c|cc|cc}
\hline
level & $\|u_h-  u\|_0 $  &rate & $ \tbar u_h- u\tbar $ &rate    \\
\hline
&\multicolumn{4}{c}{$\P_1$ element with $\P_3$ gradient ($j=3$ in \eqref{eqn:dwg})}
\\ \hline
3&   0.5015E-02 & 2.09&   0.1725E+00 & 1.01 \\
4&   0.1308E-02 & 1.94&   0.8634E-01 & 1.00 \\
5&   0.3324E-03 & 1.98&   0.4319E-01 & 1.00 \\
6&   0.8356E-04 & 1.99&   0.2160E-01 & 1.00 \\
\hline
&\multicolumn{4}{c}{$\P_2$ element with $\P_4$ gradient ($j=4$ in \eqref{eqn:dwg})}
\\ \hline
3&   0.4899E-03 & 3.56&   0.1394E-01 & 2.29 \\
4&   0.7084E-04 & 2.79&   0.3275E-02 & 2.09 \\
5&   0.1548E-04 & 2.19&   0.8049E-03 & 2.02 \\
6&   0.3853E-05 & 2.01&   0.2001E-03 & 2.01 \\
\hline
&\multicolumn{4}{c}{$\P_3$ element with $\P_5$ gradient ($j=5$ in \eqref{eqn:dwg})}
\\ \hline
2&   0.4760E-02 & 2.79&   0.3883E-01 & 2.41 \\
3&   0.2852E-03 & 4.06&   0.4779E-02 & 3.02 \\
4&   0.1755E-04 & 4.02&   0.5969E-03 & 3.00 \\
5&   0.1114E-05 & 3.98&   0.7501E-04 & 2.99 \\
\hline
\end{tabular}%
\end{table}%

\section{Conclusion}
\label{sec:conclusion}

In this paper, we studied the stabilizer-free weak Galerkin methods for a class of second order elliptic boundary value problems of divergence form and with gradient nonlinearity in the principal coefficient. With certain assumption on the nonlinear coefficient, we showed that the discrete problem has a unique solution. This was achieved by showing the associated operator satisfies certain continuity and monotonicity properties. With the help of these properties, we derived  optimal error estimates in the energy norm. We presented several numerical experiments to verify the error estimates.

From the numerical experiments in Section~\ref{sec:num}, we observed superconvergence in both $L^2$ and energy error estimates, especially on the rectangular grids (cf. Table~\ref{t1} and Table~\ref{t3}). These phenomena, as well as the effects of numerical quadrature on the approximation will be investigated in future works. 

\section*{Acknowledgments}
\label{sec:ack}
The first author was supported in part by National Science Foundation Grant DMS-1620016, and the third author was supported in part by NSF DMS-1319110.



\end{document}